\documentclass[10pt, reqno]{amsart}

\usepackage{latexsym}
\usepackage{graphics}
\usepackage{amsfonts}
\usepackage{amssymb}
\usepackage{amsmath}
\usepackage{amscd}
\usepackage[all]{xy}
\usepackage[hypertex]{hyperref}
\usepackage{multirow}

\def\3{{\ss}}
\def\2{\frac{1}{2}}
\def\4{\frac{1}{4}}
\def\8{\frac{1}{8}}
\def\x{\times}
\def\.{\cdot}

\def\bq{/\!\!/}
\def\<{\langle}
\def\>{\rangle}
\def\met{\<\; , \: \>}
\def\^{\wedge}

\newcommand{\wt}{\widetilde}
\newcommand{\wh}{\widehat}

\newcommand{\tr}{\ensuremath{\operatorname{tr}}}

\renewcommand{\Im}{\ensuremath{\operatorname{Im}}}

\renewcommand{\sec}{\ensuremath{\operatorname{sec}}}

\def\Im{\mathop{\rm Im\,}\nolimits}
\def\Re{\mathop{\rm Re\,}\nolimits}

\def\Ad{\mathop{\rm Ad}\nolimits}
\def\Span{\mathop{\rm Span\,}\nolimits}

\def\diag{\mathop{\rm diag\,}\nolimits}

\def\H{\mathord{\mathbb H}}
\def\C{\mathord{\mathbb C}}
\def\R{\mathord{\mathbb R}}
\def\Z{\mathord{\mathbb Z}}

\newcommand{\SU}{\operatorname{SU}}
\newcommand{\syp}{\operatorname{Sp}}
\newcommand{\U}{\operatorname{U}}

\def\g{\mathfrak{g}}

\def\p{\mathfrak{p}}

\def\u{\mathfrak{u}}
\def\k{\mathfrak{k}}

\def\m{\mathfrak{m}}
\newcommand{\mf}{\mathfrak}

\def\V{\mathcal{V}}

\newcommand{\mcal}{\mathcal}

\def\lra{\longrightarrow}
\def\lmt{\longmapsto}
\def\hra{\hookrightarrow}

\newcommand{\comment}[1]{}

\def\bsm{\begin{smallmatrix}}
\def\esm{\end{smallmatrix}}
\def\bpm{\begin{pmatrix}}
\def\epm{\end{pmatrix}}

\newcommand{\twoonem}[2]{\mbox{$\bpm
 #1 \\
 #2
\epm$}}

\newcommand{\twom}[4]{\mbox{$\bpm
 #1 & #2 \\
 #3 & #4
\epm$}}

\newcommand{\threem}[9]{\mbox{$\bpm
#1 & #2 & #3 \\
#4 & #5 & #6 \\
#7 & #8 & #9
\epm
$}}

\newcommand{\threemd}[3]{\threem{#1}{}{}{}{#2}{}{}{}{#3}}

\newcommand{\fourm}[4]{\mbox{$
\bpm
#1 &  &  & \\
   & #2 &  &  \\
   &  & #3 &  \\
  &  &  & #4
\epm$}}

\newcommand{\fivem}[5]{\mbox{$\bpm
 #1 &  &  & & \\
  & #2 &  & & \\
  &  & #3 & & \\
  &  & & #4 & \\
  &  & & & #5
\epm$}}

\def\beq{\begin{equation}}
\def\eeq{\end{equation}}

\def\Beq{\begin{equation*}}
\def\Eeq{\end{equation*}}

\newtheorem{thm}{Theorem}[section]

\newtheorem{prop}[thm]{Proposition}
\newtheorem{lem}[thm]{Lemma}

\theoremstyle{plain}
\newtheorem{main}{Theorem}                 

\numberwithin{equation}{section}

\theoremstyle{definition}
\newtheorem*{ack}{Acknowledgments}
\newtheorem{rmk}[thm]{Remark}

\begin{document}
\title{On the curvature of biquotients}
\author{Martin Kerin}
\address{Mathematisches Institut, WWU M\"unster, Einsteinstr. 62, 48149 M\"unster, Germany} 
\email{m.kerin@math.uni-muenster.de}
%
%

\begin{abstract}
As a means to better understanding manifolds with positive curvature, there has been much recent interest in the study of non-negatively curved manifolds which contain either a point or an open dense set of points at which all $2$-planes have positive curvature.  We study infinite families of biquotients defined by Eschenburg and Bazaikin from this viewpoint, together with torus quotients of $S^3 \x S^3$.
\end{abstract}

\maketitle

\normalsize
\thispagestyle{empty}

There exist many examples of (compact) manifolds with non-negative curvature.  All homogeneous spaces $G/H$ and all biquotients $G \bq U$ inherit non-negative curvature from the bi-invariant metric on $G$.  Additionally, it is shown in \cite{GZ} that all cohomogeneity-one manifolds, namely manifolds admitting an isometric group action with one-dimensional orbit space, admit metrics with non-negative curvature when the singular orbits are of codimension $\leq 2$.

On the other hand, the known examples with positive curvature are very sparse (see \cite{Zi1} for a survey).  Other than the rank-one symmetric spaces there are isolated examples in dimensions $6,7,12,13$ and $24$ due to Wallach \cite{Wa} and Berger \cite{Ber}, and two infinite families, one in dimension $7$ (Eschenburg spaces; see \cite{AW}, \cite{E1}, \cite{E2}) and the other in dimension $13$ (Bazaikin spaces; see \cite{Ba}).  In recent developments, two distinct metrics with positive curvature on a particular cohomogeneity-one manifold have been proposed (\cite{GVZ}, \cite{D}), while in \cite{PW2} the authors propose that the Gromoll-Meyer exotic $7$-sphere admits positive curvature.  This would be the first exotic sphere known to exhibit this property.

Unfortunately, for a simply connected manifold which admits a metric of non-negative curvature there are no known obstructions to admitting positive curvature.

In this paper we are interested in the study of manifolds which lie ``between'' those with non-negative and those with positive sectional curvature.  It is hoped that the study of such manifolds will yield a better understanding of the differences between these two classes.

Recall that a Riemannian manifold $(M, \met)$ is said to have \emph{quasi-positive curvature} (resp. \emph{almost positive curvature}) if $(M, \met)$ has non-negative sectional curvature and there is a point (resp. an open dense set of points) at which all $2$-planes have positive sectional curvature.

\begin{main}
  \label{thmA} \
  \begin{enumerate}
  \item \label{QPcurvEsch}
    All Eschenburg spaces $E^7_{p,q} = \SU(3) \bq S^1_{p,q}$ admit a metric with quasi-positive curvature.
  \item \label{APcurvEsch}
    The Eschenburg space $E^7_{p,q}$, $p=(1,1,0)$, $q=(0,0,2)$, admits almost positive curvature.
  \item \label{QPBaz}
    All Bazaikin spaces $B^{13}_{q_1, \dots, q_5} = \SU(5) \bq (\syp(2) \cdot S^1_{q_1, \dots, q_5})$ such that $0 < q_1, \dots, q_4$ admit quasi-positive curvature.
  \item \label{APBaz}
    The Bazaikin space $B^{13}_{1,1,1,1,-1}$ admits almost positive curvature.
\end{enumerate}
\end{main}

The Eschenburg spaces are defined by $E^7_{p,q} = \SU(3) \bq S^1_{p,q}$ where $p = (p_1, p_2, p_3)$, $q=(q_1,q_2,q_3) \in \Z^3$, $\sum p_i = \sum q_i$, and $S^1_{p,q}$ acts on $\SU(3)$ via
        $$z \star A = \diag(z^{p_1}, z^{p_2}, z^{p_3}) \cdot A \cdot \diag(\bar z^{q_1}, \bar z^{q_2}, \bar z^{q_3}), \ \ z \in S^1, A \in \SU(3).$$
The Bazaikin spaces are defined by $B^{13}_{q_1, \dots, q_5} = \SU(5) \bq (\syp(2) \cdot S^1_{q_1, \dots, q_5})$, where $q_1, \dots, q_5 \in \Z$, $q = \sum q_i$, and $\syp(2) \cdot S^1_{q_1, \dots, q_5} = (\syp(2) \x S^1_{q_1, \dots, q_5})/\Z_2$ acts on $\SU(5)$ via
        $$[A,z] \star B = \diag(z^{q_1}, \dots, z^{q_5}) \cdot B \cdot \diag(A, \bar z^q),$$
$z \in S^1$, $A \in \syp(2) \subset \SU(4)$, $B \in \SU(5)$.

Several large classes of examples of manifolds with almost positive curvature appear in the work of Wilking \cite{Wi}.  The only other previously known examples of manifolds with almost positive or quasi-positive curvature are given in \cite{PW1}, \cite{W}, \cite{Wi}, \cite{Ta1}, and \cite{EK}.

One of the original motivations for studying manifolds with quasi-positive curvature was the Deformation Conjecture, which stated that a complete Riemannian manifold with quasi-positive curvature admits a metric with positive curvature.  The examples in \cite{Wi} show that this conjecture is false since, for example, $\R P^3 \x \R P^2$ cannot admit positive curvature by Synge's Theorem.  However, all of Wilking's counter-examples have non-trivial fundamental group.  Therefore it is still possible that the Deformation Conjecture holds for simply connected manifolds with quasi-positive curvature.  In particular, if this conjecture were true it would follow from Wilking's examples that $S^3 \x S^2$ admits a metric with positive curvature.  This would be a counter-example to the celebrated Hopf Conjecture, which asserts that a product of spheres cannot admit positive curvature.

In \cite{PW1} the authors suggest that consideration should be given to another modification of the Deformation Conjecture, namely that a Riemannian manifold with quasi-positive curvature admits a metric with almost positive curvature.

A good illustration of the various deformations at work is the Gromoll-Meyer exotic $7$-sphere $\Sigma^7 = \syp(2) \bq \syp(1)$.  In \cite{GM} $\Sigma^7$ was shown to inherit quasi-positive curvature from the bi-invariant metric on $\syp(2)$.  It has since been shown that this metric on $\Sigma^7$ may be deformed to have almost positive curvature (see \cite{W}, \cite{EK}).  Finally, it is claimed in \cite{PW2} that one may further deform $\Sigma^7$ to admit positive curvature.

Given the dearth of examples of manifolds with positive curvature and the relative abundance of examples with non-negative curvature, it is natural to investigate the topology of known examples.  In particular, the Bazaikin spaces may be distinguished by the order $s$ of the cohomology groups $H^6 = H^8 = \Z_s$ (\cite{Ba}).  From this one can write down infinitely many positively curved Bazaikin spaces which are distinct even up to homotopy equivalence.

On the other hand, in \cite{FZ2} it is shown that there are only finitely many positively curved Bazaikin spaces for a given cohomology ring.  This statement should be viewed in the context of the Klingenberg-Sakai conjecture.  It states that there are only finitely many positively curved manifolds in a given homotopy type, and the result \cite{FZ2} raises the question of whether the conjecture is true even for cohomology.  In this context we establish the following result.

\begin{main} \
    \label{topBaz}
            There exist infinitely many pairwise non-homeomorphic Bazaikin spaces which admit quasi-positive curvature and share the same cohomology ring.
\end{main}

From Theorem \ref{topBaz} it is immediate that the Deformation Conjecture for simply connected manifolds and the cohomology Klingenberg-Sakai Conjecture cannot both be true.

If we now relax the constraint that $U$ acts freely on $G$ by allowing $U$ to act almost freely (i.e. all isotropy groups are finite), we can find the following orbifold examples:

\begin{main} \
  \label{thmB}
  \begin{enumerate}
    \item \label{APcurvEsOrb}
      All of the Eschenburg orbifolds $\SU(3) \bq S^1_{p,q}$ with $p = (p_1, p_2, p_3)$ and $q = (q_1, q_2, q_3) \in \Z^3$ satisfying
      \Beq
      \tag{$\dagger$}
      \label{cond1}
      q_1 < q_2 = p_1 < p_2 \leq p_3 < q_3
      \Eeq
      admit almost positive curvature.
    \item \label{APorbs}
      There are infinitely many orbifolds of the form $(S^3 \x S^3) \bq T^2$ admitting almost positive curvature.
  \end{enumerate}
\end{main}

We remark that there are no free $S^1_{p,q}$-actions on $\SU(3)$ satisfying condition (\ref{cond1}).  Moreover, for the $T^2$-actions on $S^3 \x S^3$ we consider, the proof that $(S^3 \x S^3) \bq T^2$ admits almost positive curvature breaks down precisely when the action is allowed to be free, namely for the quotient manifold $S^2 \x S^2$.

Among the orbifolds  $(S^3 \x S^3) \bq T^2$ there are examples with only one singular point, having isotropy group $\Z_3$, and examples with only two singular points, each with $\Z_2$ isotropy.  Some of these examples are described in Table \ref{specialisotropy}.  In \cite{FZ1} the authors get precisely the same minimal isotropy groups for positively curved six-dimensional orbifolds arising as quotients of Eschenburg spaces by an $S^1$ action.  This raises the question of whether there are positively or almost positively curved orbifolds having a unique singular point, with $\Z_2$ as its isotropy group.

The paper is organised as follows.  In Section \ref{Biqs} we recall the basic notation and techniques which will be used throughout the paper.  In Section \ref{Esch} we apply these techniques to the Eschenburg spaces in order to prove Theorem \ref{thmA}, \ref{QPcurvEsch} and \ref{APcurvEsch}, and Theorem \ref{thmB}\ref{APcurvEsOrb}.  In Section \ref{Baz} we examine the Bazaikin spaces and prove Theorem \ref{thmA}, \ref{QPBaz} and \ref{APBaz}, together with Theorem \ref{topBaz}.  Finally, in Section \ref{4orbs} we turn our attention to torus quotients of $S^3 \x S^3$ and establish Theorem \ref{thmB}\ref{APorbs}.

\section{Biquotient actions and metrics}
\label{Biqs}

In his Habilitation, \cite{E1}, Eschenburg studied biquotients in great detail.  The following section provides a review of the material in \cite{E1} which establishes the basic language, notation and results that will be used throughout the remainder of the paper.

Let $G$ be a compact Lie group, $U \subset G \x G$ a closed subgroup, and let $U$ act on $G$ via
        $$(u_1, u_2) \star g = u_1 g u_2^{-1}, \ \  g \in G, (u_1, u_2) \in U.$$
The action is free if and only if, for all non-trivial $(u_1, u_2) \in U$, $u_1$ is never conjugate to $u_2$ in $G$.  The resulting manifold is called a {\it biquotient}.

Let $K \subset G$ be a closed subgroup, $\met$ be a left-invariant, right $K$-invariant metric on $G$, and $U \subset G \x K \subset G \x G$ act freely on $G$ as above. Let $g \in G$.  Define
        \begin{align*}
            U^g_L &:= \{(g u_1 g^{-1}, u_2) \ | \ (u_1, u_2) \in U \},\\
            U^g_R &:= \{(u_1, g u_2 g^{-1}) \ | \ (u_1, u_2) \in U \},\ \ {\rm and}\\
            \wh U &:= \{(u_2, u_1) \ | \ (u_1, u_2) \in U \}.
        \end{align*}
Then $U^g_L, U^g_R$ and $\wh U$ act freely on $G$, and $G \bq U$ is isometric to $G \bq U^g_L$, diffeomorphic to $G \bq U^g_R$ (isometric if $g \in K$), and diffeomorphic to $G \bq \wh U$ (isometric if $U \subset K \x K$).

In the case of $U^g_L$ this follows from the fact that left-translation $L_g : G \lra G$ is an isometry which satisfies $g u_1 g^{-1}( L_g g') u_2^{-1} = L_g (u_1 g' u_2^{-1})$.  Therefore $L_g$ induces an isometry of the orbit spaces $G\bq U$ and $G\bq U^g_L$.  Similarly we find that $R_{g^{-1}}$ induces a diffeomorphism between $G\bq U$ and $G \bq U^g_R$, which is an isometry if $g \in K$.

Consider now $\wh U$.  The actions of $U$ and $\wh U$ are equivariant under the diffeomorphism $\tau : G \lra G$, $\tau(g) := g^{-1}$.  That is, $u_1 \tau(g) u_2^{-1} = \tau(u_2 g u_1^{-1})$.  Notice that this is an isometry only if $U \subset K \x K$.  In general $G\bq U$ and $G\bq \wh U$ are therefore diffeomorphic but not isometric.

Suppose $\pi: M^n \lra N^{n-k}$ is a Riemannian submersion.  The O'Neill formula for Riemannian submersions implies that $\pi$ is curvature non-decreasing.  Therefore if $\sec_M \geq 0$ then $\sec_N \geq 0$, and zero-curvature planes on $N$ lift to horizontal zero-curvature planes on $M$.  In general, because of the Lie bracket term in the O'Neill formula, the converse is not true, namely horizontal zero-curvature planes in $M$ cannot be expected to project to zero-curvature planes on $N$.  However, we will see at the end of this section that in many situations we have $\sec_N (X, Y) = 0$ if and only if $\sec_M (\wt X, \wt Y) = 0$, where $\wt X$ denotes the horizontal lift to $T_p M$ of $X \in T_{\pi(p)} N$.

Let $K \subset G$ be Lie groups, $\k \subset \g $ the corresponding Lie algebras, and $\met_0$ a bi-invariant metric on $G$.  Note that $(G, \met_0)$ has $\sec \geq 0$, and $\sigma = \Span \{X,Y\}$ has $\sec(\sigma) = 0$ if and only if $[X,Y] = 0$.  We can write $\g = \k \oplus \p$ with respect to $\met_0$.  Given $X \in \g$ we will always use $X_\k$ and $X_\p$ to denote the $\k$ and $\p$ components of $X$ respectively.

Recall that
        $$G \cong (G \x K)/ \Delta K$$
via $(g,k) \lmt g k^{-1}$, where $\Delta K$ acts diagonally on the right of $G \x K$.  Thus we may define a new left-invariant, right $K$-invariant metric $\met_1$ (with $\sec \geq 0$) on $G$ via the Riemannian submersion
        \begin{align*}
            (G \x K, \met_0 \oplus t \met_0 |_\k) &\lra (G, \met_1)\\
            (g,k) &\lmt g k^{-1},
        \end{align*}
where $t>0$ and
        \beq
        \label{met1}
            \met_1 = \met_0 |_\p + \lambda \met_0 |_\k, \ \ \lambda = \frac{t}{t+1} \in (0,1).
        \eeq
In particular notice that
        $$\< X, Y \>_1 = \< X, \Phi(Y) \>_0, \ \ \ {\rm where} \ \Phi(Y) = Y_\p + \lambda Y_\k, \ \lambda \in (0,1).$$
It is clear that the metric tensor $\Phi$ is invertible with inverse given by $\Phi^{-1}(Y) = Y_\p + \frac{1}{\lambda} Y_\k$.

In the special case that $(G,K)$ is a symmetric pair we have the following useful lemma.

\begin{lem}[Eschenburg]
\label{Eschlem}
Let $(G,K)$ be a symmetric pair.  Then a plane $\sigma = \Span \{\Phi^{-1}(X),\Phi^{-1}(Y)\}$ has $\sec(\sigma) = 0$ with respect to $\met_1$ if and only if
        $$0 = [X,Y] = [X_\k,Y_\k] = [X_\p,Y_\p].$$
\end{lem}

Recall that for a bi-invariant metric we get $\sec(X,Y) = 0$ if and only if $[X,Y]=0$.  For our left-invariant metric $\met_1$ we have two extra conditions which must be satisfied for a plane to have zero-curvature, and hence we may have reduced the number of such planes.

Suppose we have a biquotient $G \bq U$, where $U \subset G \x K \subset G \x G$ and $G$ is equipped with a left-invariant, right $K$-invariant metric constructed as above.  Then $U$ acts by isometries on $G$ and therefore the submersion $G \lra G \bq U$ induces a metric on $G \bq U$ from the metric on $G$.  By our discussion of the O'Neill formula above we know that a zero-curvature plane on $G \bq U$ with respect to the induced metric must lift to a horizontal zero-curvature plane in $G$.

In order to determine what it means for a plane to be horizontal we must first determine the vertical distribution on $G$.  Note that this is independent of the choice of left-invariant metric on $G$.  The fibre through a particular point $g \in G$ is
        $$F_g := \{u_1 g u_2^{-1} \ | \ (u_1, u_2) \in U\}.$$
If $u(t) := \exp(tX)$, where $X = (X_1, X_2) \in \mf{u}$ and $\mf{u}$ is the Lie algebra of $U$, then $u_1(t)\: g \: u_2(t)^{-1}$ is a curve in $F_g$ and
        $$\frac{d}{dt} \: u_1(t) \: g \: u_2(t)^{-1} \; \Big|_{t=0} = (R_g)_* X_1 - (L_g)_* X_2 =: v_g(X)$$
is a typical vertical vector.  The vector field $v(X)$ on $G$ defined in such a way is the Killing vector field associated to $X$.  Since $G$ is equipped with a left-invariant metric we may shift the vertical space $V_g = \{v_g(X) \ | \ X \in \mf{u} \}$ to the identity $e \in G$ by left-translation and get
        $$\mcal{V}_g := (L_{g^{-1}})_* V_g$$
whose elements are of the form
        $$(L_{g^{-1}})_* v_g(X) = \Ad_{g^{-1}} X_1 - X_2.$$
We may therefore define the horizontal subspace at $g \in G$ by
        $$\mcal{H}_g := \mcal{V}_g^\perp.$$
It is important to remark that the horizontal subspace at $g$ depends on the choice of left-invariant metric as it is defined by $\mcal{V}_g^\perp$, where we are taking the orthogonal complement with respect to the metric on $G$.

Suppose $G$ is equipped with a bi-invariant metric.  Eschenburg \cite{E1} provides some sufficient conditions under which a horizontal zero-curvature in $G$ projects to a zero-curvature plane in a biquotient $G \bq U$.  Wilking \cite{Wi} has generalised this to show that, given any biquotient submersion $G \lra G \bq U$, a horizontal zero-curvature plane in $G$ must always project to a zero-curvature plane in $G \bq U$.  Tapp \cite{Ta2} has recently generalised this result even further.

\begin{thm}[Tapp]
\label{tapp}
Suppose $G$ is a compact Lie group equipped with a bi-invariant metric and that $G \lra B$ is a Riemannian submersion.  Then a horizontal zero-curvature plane in $G$ projects to a zero-curvature plane in $B$.
\end{thm}

It follows immediately from the above theorem that if we have a pair of Riemannian submersions $G \lra M \lra B$, where $G$ is equipped with a bi-invariant metric, then a horizontal zero-curvature plane in $M$ must project to a zero-curvature plane in $B$.

Notice that in the metric construction on $G \bq U$ described above we have Riemannian submersions $G \x K \lra G \lra G \bq U$ where $G \x K$ is equipped with a bi-invariant metric.  Therefore in order to find zero-curvature planes in $(G, \met_1) \bq U$ we may concentrate exclusively on the more tractable problem of finding horizontal zero-curvature planes in $G$.

\section{Eschenburg Spaces}
\label{Esch}

Recall that the Eschenburg spaces are defined as $E^7_{p,q} := \SU(3) \bq S^1_{p,q}$, where $p = (p_1, p_2, p_3)$, $q = (q_1, q_2, q_3) \in \Z^3$, $\sum p_i = \sum q_i$, and $S^1_{p,q}$ acts on $\SU(3)$ via
        $$z \star A = \threemd{z^{p_1}}{z^{p_2}}{z^{p_3}} A \threemd{\bar z^{q_1}}{\bar z^{q_2}}{\bar z^{q_3}}, \ \ A \in \SU(3), z \in S^1.$$
The action is free if and only if
        \beq
        \label{freeness}
            (p_1 - q_{\sigma(1)}, p_2 - q_{\sigma(2)}) = 1 \ \  \textrm{for all} \ \ \sigma \in S_3.
        \eeq
Let $K = \U(2) \hookrightarrow G = \SU(3)$ via
        $$A \in \U(2) \lmt \twom{A}{}{}{\alpha} \in \SU(3), \ \  \alpha = \overline{\det(A)}.$$
$(G,K)$ is a rank one symmetric pair.  Let $\met_0$ be the bi-invariant metric on $G$ given by $\<X,Y\>_0 = - \Re \tr (XY)$.  We can write $\mf{su}(3)=\g = \k \oplus \p$ with respect to $\met_0$.  We define a new left-invariant, right $K$-invariant metric $\met_1$ (with $\sec \geq 0$) on $G$ as in (\ref{met1}) and may therefore apply Lemma \ref{Eschlem}.

From Section \ref{Biqs} we know that, for the $S^1_{p,q}$-action, permuting the $p_i$'s and permuting $q_1, q_2$ are isometries, while permuting the $q_i$'s and swapping $p$, $q$ are diffeomorphisms.

Let
        $$Y_1 := i \threemd{-2}{1}{1}, \ \  Y_3 := i \threemd{1}{1}{-2} \in \g = \mf{su}(3).$$
Using Lemma \ref{Eschlem} Eschenburg \cite{E1} showed that in this special case we can easily determine when a plane in $\g$ has zero-curvature.
\begin{lem}[Eschenburg]
\label{lemY1Y3}
$\sigma = \Span \{X,Y\} \subset \mf{su}(3)$ has $\sec(\sigma) = 0$ with respect to $\met_1$ if and only if either $Y_3 \in \sigma$, or $\Ad_k Y_1 \in \sigma$ for some $k \in K$.
\end{lem}

We may apply this lemma in order to discuss when an Eschenburg space $E^7_{p,q}$ admits positive curvature.  While this is well-known, in our proof we compute explicit equations ((\ref{Y3horiz}) and (\ref{Y1horiz})) for the existence of zero-curvature planes in $E^7_{p,q}$, which we use to prove Theorem \ref{thmA}\ref{QPcurvEsch} and \ref{APcurvEsch}.

\begin{thm}[Eschenburg]
\label{poscurv}
$E^7_{p,q} := (\SU(3), \met_1) \bq S^1_{p,q}$ has positive curvature if and only if
        \beq
        \label{poscurvcond}
            q_i \not\in [\underline{p}, \overline{p}] \ \ {\rm for} \; i=1,2,3,
        \eeq
where $\underline{p} := \min \{p_1, p_2, p_3 \}$, $\overline{p} := \max \{p_1, p_2, p_3 \}$.
\end{thm}

\begin{proof}
We will first prove that the condition (\ref{poscurvcond}) gives positive curvature.  By Lemma \ref{lemY1Y3} we need only show that we may choose an ordering on the $q_i$'s so that $Y_3$ and $\Ad_k Y_1$ are never horizontal.

Let $P = i \diag(p_1, p_2, p_3)$ and $Q = i \diag(q_1, q_2, q_3)$.  From our discussion of vertical spaces in Section \ref{Biqs} we find that the vertical subspace at $A = (a_{ij}) \in \SU(3)$ is
        $$\V_A = \left\{t \: v_A \ \Big| \; t \in \R, v_A:=\Ad_{A^*} P - Q \right\},$$
where $A^* = \bar A^t$.  Notice that $Y_3 \in \k$.  Thus $0 = \<v_A, Y_3\>_1$ if and only if $0 = \<v_A, Y_3\>_0$.  Now, since $\<X,Y\>_0 = - \Re \tr (XY)$,
        \beq
        \label{Y3horiz}
            0 = \<v_A, Y_3\>_1 \iff \sum_{j=1}^3 |a_{j3}|^2 p_j = q_3.
        \eeq
Similarly, for $\Ad_k Y_1$, $k \in K$, we find
        \beq
        \label{Y1horiz}
            0 = \<v_A, \Ad_k Y_1\>_1 \iff \sum_{j=1}^3 |(Ak)_{j1}|^2 p_j = |k_{11}|^2 q_1 + |k_{21}|^2 q_2.
        \eeq

Now, since $q_i \not\in [\underline{p}, \overline{p}], i=1,2,3$, and $\sum p_j = \sum q_j$, we know that two of the $q_i$'s must lie on one side of $[\underline{p}, \overline{p}]$, and one on the other.  We reorder and relabel the $q_i$'s so that $q_1, q_2$ lie on the same side of $[\underline{p}, \overline{p}]$.  Since $A$ and $k$ are both unitary we therefore have that there are no solutions to either (\ref{Y3horiz}) or (\ref{Y1horiz}).  Hence $E^7_{p,q}$ has positive curvature.

For the converse suppose that $E^7_{p,q}$ has positive curvature.  If $q_i \in [\underline{p}, \overline{p}]$ for some $i=1,2,3$ then by continuity there exists a solution to either (\ref{Y3horiz}) or (\ref{Y1horiz}), and hence either $Y_3$ or $\Ad_k Y_1$ is horizontal.  By Lemma \ref{lemY1Y3}, since the orbits of $S^1_{p,q}$ are one-dimensional, we can always find another horizontal vector $X$ which, together with either $Y_3$ or $\Ad_k Y_1$, will span a zero-curvature plane.  Theorem \ref{tapp} then implies that this horizontal zero-curvature plane must project to a zero-curvature plane in $E^7_{p,q}$ and so we have a contradiction.\qed
\end{proof}


We will now discuss some new results on the curvature of general Eschenburg spaces.

\begin{thm}
All Eschenburg spaces admit a metric with quasi-positive curvature.
\end{thm}

\begin{proof}
We need to find a point in $\SU(3)$ at which there are no horizontal zero-curvature planes, i.e. at which $Y_3$ and $\Ad_k Y_1$ are not horizontal.

Let $A \in \SU(3)$ be a diagonal matrix.  Thus equation (\ref{Y1horiz}) becomes
        \begin{align*}
            &{} |k_{11}|^2 p_1 + |k_{21}|^2 p_2 = |k_{11}|^2 q_1 + |k_{21}|^2 q_2 \\
            &\iff (p_1 - q_1)|k_{11}|^2 + (p_2 - q_2)|k_{21}|^2 = 0.
        \end{align*}
Therefore, if
        \beq
        \label{quasicond}
            (p_1 - q_1)(p_2 - q_2) > 0
        \eeq
there is no $k \in K$ satisfying (\ref{Y1horiz}), i.e. $\Ad_k Y_1$ is not horizontal at $A$.

On the other hand, equation (\ref{Y3horiz}) becomes $p_3 = q_3$.  However, (\ref{quasicond}), together with $\sum p_i = \sum q_i$, implies that $p_3 \neq q_3$, i.e. that $Y_3$ is not horizontal at $A$.

Thus, if (\ref{quasicond}) holds, then $E^7_{p,q}$ has $\sec > 0$ at $[A]$, where $A \in \SU(3)$ is diagonal.

Recall the freeness condition (\ref{freeness}) and that permuting the $p_i$'s and $q_j$'s are diffeomorphisms.  Therefore, as long as there is no $i \in \{1,2,3\}$ such that $p_i = q_j$ for all $j \in \{1,2,3\}$, we may always reorder and relabel the $p_i$'s and $q_j$'s such that (\ref{quasicond}) holds.

By (\ref{freeness}), the only Eschenburg space satisfying the condition ``\emph{there is an $i \in \{1,2,3\}$ such that $p_i = q_j$ for all $j \in \{1,2,3\}$}'' is the Aloff-Wallach space $W_{-1,1} := E^7_{p,q}, p=(-1,1,0), q=(0,0,0)$.  However, Wilking \cite{Wi} has shown that $W_{-1,1}$ admits a metric with almost positive curvature, and so we are done.\qed
\end{proof}

The special subfamily $E^7_n := E^7_{p,q}, p = (1,1,n), q = (0,0,n+2),$ admits a cohomogeneity-one action by $\SU(2) \x \SU(2)$.  These cohomogeneity-one Eschenburg spaces are discussed in great detail in \cite{GSZ}.  We may assume that $n \geq 0$ since $E^7_n \cong E^7_{-(n+1)}$.  By Theorem \ref{poscurv}, $n>0$ implies that $E^7_n$ admits a metric with positive curvature.

\begin{thm}
\label{E0almostpos}
$E^7_0$ admits a metric with almost positive curvature.
\end{thm}

\begin{proof}
Given $p = (1,1,0)$ and $q = (0,0,2)$, equations (\ref{Y3horiz}) and (\ref{Y1horiz}) become
        \beq
        \label{EQ1}
        2 = |a_{13}|^2 + |a_{23}|^2
        \eeq
and
        \begin{align}
            \nonumber  |(Ak)_{11}|^2 + |(Ak)_{21}|^2 &= 0\\
            \nonumber \iff (Ak)_{11} = (Ak)_{21} &= 0\\
            \label{EQ2} \iff \twom{a_{11}}{a_{12}}{a_{21}}{a_{22}}  \twoonem{k_{11}}{k_{21}} &= 0
        \end{align}
respectively.  Since $A \in \SU(3)$ it is clear that (\ref{EQ1}) cannot be satisfied.  Since $k \in K = \U(2)$, we are only interested in solutions $\left(\bsm k_{11} \\ k_{21} \esm\right) \neq 0$.  This occurs if and only if
        $$\det \twom{a_{11}}{a_{12}}{a_{21}}{a_{22}} = 0,$$
which defines a codimension two sub-variety $\Omega \subset \SU(3)$ of points with horizontal zero-curvature planes.  Moreover it is easy to check that $\Omega$ is a smooth sub-variety.  Since the equation which defines $\Omega$ is preserved under the $S^1_{p,q}$-action, $E^7_0$ has almost positive curvature and points in $E^7_0$ with zero-curvature planes form a smooth codimension two submanifold.\qed
\end{proof}

We may fix a particular metric on $E^7_{p,q}$ by choosing $p_1 \leq p_2 \leq p_3$ and $q_1 \leq q_2 \leq q_3$.  Therefore Eschenburg's positive curvature condition is
        \beq
        \label{poscond}
            q_1 \leq q_2 < p_1 \leq p_2 \leq p_3 < q_3 \ \ {\rm or} \ \ q_1 < p_1 \leq p_2 \leq p_3 < q_2 \leq q_3.
        \eeq
It is natural to ask what happens when $q_2 = p_1$ or $q_2 = p_3$, which we refer to as the ``boundary'' of the positive curvature condition.

\begin{lem}
\label{Eschmnfds}
The only free $S^1_{p,q}$-actions on $\SU(3)$ satisfying $q_2 = p_1$ or $q_2 = p_3$ are, up to diffeomorphism,
        \begin{enumerate}
        \item
        \label{Action1}
            $p=(0,0,0)$ and $q=(- 1,0, 1)$, and

        \item
        \label{Action2}
            $p=(0,1,1)$ and $q=(0,0, 2)$.
        \end{enumerate}
\end{lem}

\begin{proof}
We need only consider the case $q_2 = p_1$, since it is clear that $E^7_{p,q}$ is diffeomorphic to $E^7_{p',q'}$, where $p' = (-p_3, -p_2, -p_1)$, $q' = (-q_3, -q_2, -q_1)$.  Since $\Delta S^1$ commutes with $\SU(3)$ we may write $p=(0,p_2,p_3)$ and $q=(q_1,0, q_3)$ without loss of generality.  By considering the freeness condition (\ref{freeness}) and the ordering of our integers we must have $p=(0,p_2,p_3)$ and $q=(p_2 - 1,0, p_2 + 1)$.  Since $\sum p_i = \sum q_i$ we have $p=(0,p_2,p_2)$ and $q=(p_2 - 1,0, p_2 + 1)$.  Hence, since we have assumed that our triples of integers are ordered, i.e. $0 \leq p_2$ and $p_2 - 1 \leq 0 \leq p_2 + 1$, either $p_2=0$ or $p_2=1$ as desired.\qed
\end{proof}

Notice that the resulting manifolds are diffeomorphic to the exceptional Aloff-Wallach space $W^7_{-1,1}$ and the exceptional cohomogeneity-one Eschenburg space $E^7_0$ for actions \ref{Action1} and \ref{Action2} respectively.  As previously discussed, both manifolds have been shown to admit metrics with almost positive curvature.  Note also that action \ref{Action1} is the action given by $q_1 < q_2 = p_1 = p_2 = p_3 < q_3$, and action \ref{Action2} is the action given by $q_1 = q_2 = p_1 < p_2 = p_3 < q_3$.  Even though there are no other manifolds on the boundary of the positive curvature condition, we can prove the following:

\begin{thm}
\label{ofldalmostpos}
All orbifolds $E^7_{p,q}$ satisfying
        \beq
        \label{obfdalmostposcond}
            q_1 < q_2 = p_1 < p_2 \leq p_3 < q_3 \ \ {\rm or} \ \ q_1 < p_1 \leq p_2 < p_3 = q_2 \leq q_3,
        \eeq
admit almost positive curvature.
\end{thm}

\begin{proof}
As in the proof of Lemma \ref{Eschmnfds}, we need only consider
        \beq
        \label{cond}
            q_1 < q_2 = p_1 < p_2 \leq p_3 < q_3,
        \eeq
since $E^7_{p,q}$ is diffeomorphic to $E^7_{p',q'}$, where as before $p' = (-p_3, -p_2, -p_1)$ and $q' = (-q_3, -q_2, -q_1)$.

Notice that (\ref{cond}) implies that (\ref{Y3horiz}) has no solutions, since $q_3 > p_i$ for all $i = 1,2,3$.

Consider for a moment the more general case of Eschenburg spaces $E^7_{p,q}$ given by $q_1 < p_1 \leq q_2 < p_2 \leq p_3 < q_3$, hence not admitting positive curvature.  Suppose that there is a $k \in K$ such that $\Ad_k Y_1$ is horizontal at some $A \in \SU(3)$.  Then (\ref{Y1horiz}) implies that
        $$p_1 \leq \sum_{j=1}^3 |(Ak)_{j1}|^2 p_j = |k_{11}|^2 q_1 + |k_{21}|^2 q_2 \leq q_2.$$
Since $|k_{11}|^2 + |k_{21}|^2 = 1$ we thus have
        $$p_1 \leq |k_{11}|^2 (q_1 - q_2) + q_2 \leq q_2 \ \ {\rm and } \ \ p_1 \leq q_1 + |k_{21}|^2 (q_2 - q_1) \leq q_2,$$
which are equivalent to
        $$0 \leq |k_{11}|^2 \leq \frac{q_2 - p_1}{q_2 - q_1} \ \ {\rm and } \ \ \frac{p_1 - q_1}{q_2 - q_1} \leq |k_{21}|^2 \leq 1.$$
In particular, when the hypothesis of the theorem is satisfied, namely $p_1 = q_2$, we get $|k_{11}|^2 = 0$ and $|k_{21}|^2 = 1$, i.e.
        $$k = \threem{0}{k_{12}}{0}{k_{21}}{0}{0}{0}{0}{- \overline{k_{12} k_{21}}} \in K=\U(2).$$
Hence (\ref{Y1horiz}) becomes
        \begin{align*}
            &{} |a_{12}|^2 p_1 + |a_{22}|^2 p_2 + |a_{32}|^2 p_3 = q_2 = p_1\\
            &\iff |a_{22}|^2 (p_2 - p_1) + |a_{32}|^2 (p_3 - p_1) = 0, \ \ {\rm since} \ \ A \in \SU(3)\\
            &\iff a_{22} = a_{32} = 0, \ \ {\rm since} \ \ p_1 < p_2 \leq p_3\\
            &\iff A = \threem{0}{a_{12}}{0}{a_{21}}{0}{a_{23}}{a_{31}}{0}{a_{33}} \in \SU(3).
        \end{align*}
The set of such $A \in \SU(3)$ is preserved under the $S^1_{p,q}$-action, hence projects to a set of measure zero in $E^7_{p,q}$.  Therefore $E^7_{p,q}$ has almost positive curvature.\qed
\end{proof}

In \cite{FZ1} it is shown that the set $\left\{S^1_{p,q} \star A \ \Big| \ A = \left(\bsm 0 & a_{12} & 0 \\
                                a_{21} & 0 & a_{23} \\
                                a_{31} & 0 & a_{33} \esm \right) \right\} \subset E^7_{p,q}$
describes a totally geodesic lens space.  For $q_1 < q_2 = p_1 < p_2 \leq p_3 < q_3$ we know from the proof of Theorem \ref{ofldalmostpos} that these are the only points admitting zero-curvature planes.  The problem of determining how large the set of zero-curvature planes is at each point of this lens space is equivalent to determining how large the set of horizontal zero-curvature planes is at each
$A = \left(\bsm 0 & a_{12} & 0 \\
           a_{21} & 0 & a_{23} \\
           a_{31} & 0 & a_{33} \esm \right) \in \SU(3)$.

\begin{prop}
If $q_1 < q_2 = p_1 < p_2 \leq p_3 < q_3$ then there is a one-dimensional family of horizontal zero-curvature planes at each point
$A = \left(\bsm 0 & a_{12} & 0 \\
           a_{21} & 0 & a_{23} \\
           a_{31} & 0 & a_{33} \esm \right) \in \SU(3)$.
\end{prop}

\begin{proof}
Recall we have shown in the proof of Theorem \ref{ofldalmostpos} that $Y_3$ is never horizontal, and $\Ad_k Y_1$ being horizontal at $A$ implies that
        $$k = \bpm 0 & k_{12} & 0 \\
                k_{21} & 0 & 0 \\
                0& 0 & - \overline{k_{12} k_{21}} \epm \in K=\U(2).$$
Hence $\Ad_k Y_1 = Y_2 := i \left( \bsm 1 & & \\
                                        & -2 & \\
                                        & & 1 \esm \right)$.\\
Let $Y = \Phi^{-1} (Y_2)$, where $\< X, Z\>_1 = \< X, \Phi(Z)\>_0$.  Let $X \in \mathcal{H}_A$ be such that $\Span \{X,Y\}$ is a horizontal zero-curvature plane.  Then, by Lemma \ref{Eschlem} and since $\Phi^{-1} (Y_2) = \frac{1}{\lambda} Y_2 \in \k$, $[X,Y] = [X_\k, Y] = 0$, which is equivalent to
        \begin{align*}
            &{} [X,Y_2] = [X_\k, Y_2] = 0 \\
            &\iff [X,Y_2] = 0\\
            &\iff X = \threem{is}{0}{x}{0}{it}{0}{-\bar x}{0}{-i(s+t)},
        \end{align*}
where $s,t \in \R$, $x \in \C$.  We may assume without loss of generality that $\<X,Y\>_1 = 0$.  Hence
        $$X = \threem{is}{0}{x}{0}{0}{0}{-\bar x}{0}{-is}.$$
The set of such $X$ is $3$-dimensional.  We also require that $X$ is horizontal, i.e. $\<X, \Ad_{A^*} P - Q\>_1 = 0$, and without loss of generality we may assume that $||X||^2 = 1$.  Thus, for each
$A = \left(\bsm 0 & a_{12} & 0 \\
           a_{21} & 0 & a_{23} \\
           a_{31} & 0 & a_{33} \esm \right) \in \SU(3)$
there is a one-dimensional family of horizontal zero-curvature planes $\Span\{X,Y\}$.\qed
\end{proof}

\section{Bazaikin Spaces}
\label{Baz}

The proof of positive curvature on an infinite subfamily of the Bazaikin spaces (given in \cite{Zi2}, \cite{DE}) follows from essentially the same techniques as in the case of the Eschenburg spaces.  A slight modification of this argument allows us to prove Theorem \ref{thmA}\ref{QPBaz} and \ref{APBaz}.

Recall that the Bazaikin spaces are defined as
        $$B^{13}_{q_1, \dots, q_5} := \SU(5) \bq \syp(2) \cdot S^1_{q_1, \dots, q_5},$$
where $q_1, \dots, q_5 \in \Z$, and
        $$ \syp(2) \cdot S^1_{q_1, \dots, q_5} = (\syp(2) \x S^1_{q_1, \dots, q_5})/\Z_2,\ \ \ \Z_2 = \{\pm(1, I)\},$$
acts effectively on $\SU(5)$ via
        $$[A,z] \star B = \threemd{z^{q_1}}{\ddots}{z^{q_5}}  B  \twom{\hat A}{}{}{\bar z^q},$$
with $z \in S^1$, $A \in \syp(2) \hra \SU(4)$, $B \in \SU(5)$, and $q = \sum q_i$.  We recall that
        \begin{align*}
            \syp(2) &\hra \SU(4)\\
            A = S + T j &\lmt \hat A = \twom{S}{T}{- \bar T}{\bar S}.
        \end{align*}
It is not difficult to show that the action of $\syp(2) \cdot S^1_{q_1, \dots, q_5}$ is free if and only all $q_1, \dots, q_5$ are odd and
        \beq
        \label{freeBaz}
            (q_{\sigma(1)} + q_{\sigma(2)}, q_{\sigma(3)} + q_{\sigma(4)}) = 2 \ \ \textrm{for all} \ \ \sigma \in S_5.
        \eeq
Let $G = \SU(5) \supset K = \U(4)$, where $K \hra G$ via
        $$A \lmt \twom{A}{}{}{\overline{\det A}}.$$
Then $(G,K)$ is a rank one symmetric pair, with Lie algebras $(\g, \k)$.  With respect to the bi-invariant metric $\<X,Y\>_0 = - \Re \tr XY$ we may write $\g = \k \oplus \p$.  Define a metric, $\met_1$, on $G$ as in (\ref{met1}) which is left-invariant and right $K$-invariant.  In particular we have $\< X, Y \>_1 = \<X, \Phi(Y) \>_0$, where $\Phi(Y) = Y_\p + \lambda Y_\k$, $\lambda \in (0,1)$.  By Lemma \ref{Eschlem} we know that a plane $\sigma = \Span\{\Phi^{-1}(X),\Phi^{-1}(Y)\} \subset \g$ has zero-curvature with respect to $\<\ ,\>_1$ if and only if
        $$0=[X,Y]=[X_\p,Y_\p]=[X_\k,Y_\k].$$
It is clear that the action of $U := \syp(2) \cdot S^1_{q_1, \dots, q_5}$ is by isometries and we therefore get an induced metric on $B^{13}_{q_1, \dots, q_5} = G \bq U$.

Let $Q = i \diag(q_1, \dots, q_5)$. From our discussion of vertical subspaces in Section \ref{Biqs}, the vertical subspace at $A \in \SU(5)$ with respect to the $U$-action may be written as
        $$\V_A = \left\{ t \Ad_{A^*} Q - \twom{X}{}{}{i t q} \ \Big| \ t \in \R, X \in \mf{sp}(2) \subset \mf{su}(4) \right\}$$
where $A^* = \bar A^t$.  Our aim is to determine when zero-curvature planes with respect to $\met_1$ are horizontal at $A \in \SU(5)$.  A vector $\Phi^{-1} (X)$ is orthogonal to $\V_A$ with respect to $\met_1$ if and only if
        \beq
        \label{BazHorizCond}
            \left\< X, \Ad_{A^*} Q - \left(\bsm 0 &&&& \\ & 0 &&& \\ && 0 && \\ &&& 0 & \\ &&&& i q \esm\right) \right\>_0 = 0 \ \ \ {\rm and} \ \ \ X \perp_0 \mf{sp}(2) \subset \mf{su}(4),
        \eeq
where $\perp_0$ denotes orthogonality with respect to $\met_0$.

\begin{lem}
\label{Baz0curv}
A $2$-plane $\sigma = \Span\{\Phi^{-1}(X),\Phi^{-1}(Y)\} \subset \g$ is a horizontal zero-curvature plane with respect to $\met_1$ if and only if either
        $$W_1 := \diag( i, i, i , i, -4 i) \ \ {\rm or} \ \ W_2 := \Ad_k \diag(2i,-3i,2i,-3i, 2i),$$
for some $k \in \syp(2)$, is in $\sigma$ and is horizontal.
\end{lem}
\begin{proof}
Suppose that the plane $\sigma = \Span\{\Phi^{-1}(X),\Phi^{-1}(Y)\}$ has zero-curvature with respect to $\met_1$.  Then, since $[X_\p,Y_\p] = 0$ by Lemma \ref{Eschlem}, we may assume without loss of generality that $Y_\p = 0$, i.e $X = X_\p + X_\k$, $Y = Y_\k$.

If we also have $X_\p = 0$, then $X, Y \in \k$.  Notice that $\k = \mf{z} \oplus \mf{sp}(2) \oplus \m$, where $\mf{z} \perp \mf{su}(4)$ is the centre of $\k$, generated by $\diag( i, i, i , i, -4 i)$, and $\m = \mf{sp}(2)^\perp \subset \mf{su}(4)$.  But we have assumed that $X, Y \perp_0 \mf{sp}(2)$.  Thus $X, Y \in \mf{z} \oplus \m$, and $[X,Y] = 0$ if and only if $[X_\m, Y_\m] = 0$.  Now $\SU(4) = Spin(6)$, $\syp(2) = Spin(5)$ and $(\SU(4), \syp(2))$ is a rank one symmetric pair.  Therefore $X_\m, Y_\m$ must be linearly dependent and we may assume without loss of generality that $X = X_\m$, $Y = Y_{\mf{z}}$.  Then $\mf{z} \subset \sigma$, i.e. $W_1 = \diag(i, i, i , i, -4 i) \in \sigma$.

We now note that $W_1$ being horizontal is not only a necessary condition for $\sigma \subset \k$ to be a horizontal zero-curvature plane, but also sufficient for the existence of such a plane as, by counting dimensions, we may always find a vector $X \in \m$ such that $\sigma = \Span\{\Phi^{-1}(X),\Phi^{-1}(W_1)\}$ is a horizontal zero-curvature plane.

On the other hand, suppose now that $X_\p \neq 0$.  Then the conditions for zero-curvature become $0=[X_\p,Y_\k]=[X_\k,Y_\k]$.  Suppose that
        $$X_\p = \twom{0}{x}{- \bar x^t}{0}, \ \ Y = Y_\k = \twom{Z}{}{}{- \tr Z},$$
where $x \in \C^4$ and $Z \in \u(4) = \mf{z} \oplus \mf{su}(4)$.  Then $0=[X_\p,Y_\k]$ if and only if $Z x = -(\tr Z) x$.  Let $Z = i t I + Z' \in \mf{z} \oplus \mf{su}(4)$, $t \in \R$.  Since it is required that $Y \perp \mf{sp}(2)$ we have $Z' \perp \mf{sp}(2) \subset \mf{su}(4)$.  Recall that $\SU(4) = Spin(6)$, $\syp(2) = Spin(5)$.  Therefore $\SU(4)/\syp(2) = S^5$ and, since $\syp(2) = Spin(5)$ acts transitively on distance spheres in $\m = \mf{sp}(2)^\perp \subset \mf{su}(4)$, we may write
        $$Z' = k \fourm{is}{-is}{is}{-is} k^{-1}, \ \ \ k \in \syp(2).$$
This in turn implies that $Z$ may be written as
        $$Z = k \fourm{i(t+s)}{i(t-s)}{i(t+s)}{i(t-s)} k^{-1}, \ \ \ k \in \syp(2).$$
But we established above that $- \tr Z = - 4it$ is an eigenvalue of $Z$.  Therefore either $-4t = t+s$ or $-4t = t-s$, i.e. $s = -5t$ or $s = 5t$.  Thus we have shown that $Y$ must be conjugate by an element of $\syp(2)$ to either $\diag( -4it, 6it, -4it, 6it, -4it)$ or $\diag( 6it, -4it, 6it, -4it, -4it)$, and so up to scaling we have
        $$Y = k \fivem{2i}{-3i}{2i}{-3i}{2i} k^{-1}, \ \ \ k \in \syp(2) \subset \SU(4) \subset \SU(5).$$
Notice that $\Phi^{-1}(Y)$ is a multiple of $Y$ and so we have $Y \in \sigma$.  Conversely, if such a vector $Y$ is horizontal it is not difficult to find a complementary vector $X$ such that $\sigma = \Span\{\Phi^{-1}(X),\Phi^{-1}(Y)\}$ is a horizontal zero-curvature plane.  Set $X_\k = 0$.  $X$ is therefore automatically orthogonal to $\mf{sp}(2)$ and it remains to choose $X_\p$ such that $X$ satisfies the first condition of (\ref{BazHorizCond}), namely that $X$ is orthogonal to a one-dimensional subspace.  A choice of appropriate $X_\p$ is equivalent to choosing an eigenvector for $Z$ above.  The set of such eigenvectors has dimension $>1$ and so we may thus choose $X_\p$ such that $X$ has the desired properties.\qed
\end{proof}

At this stage of the positive curvature argument in \cite{Zi2} a lemma due to Eschenburg is applied to avoid direct computations.  However, to prove Theorem \ref{thmA}\ref{QPBaz} and \ref{APBaz} we need to perform these computations in order to derive some equations which may be exploited in a similar manner to that employed in Section \ref{Esch}.

\begin{lem}
\label{Bazeqns}
The vectors
        $$W_1 = \diag( i, i, i , i, -4 i) \ \ \ {\rm and} \ \ \ W_2 = \Ad_k \diag(2i,-3i,2i,-3i,2i),$$
$k \in \syp(2)$, are horizontal with respect to $\met_1$ at $A = (a_{ij}) \in \SU(5)$ if and only if
        \begin{align}
        \label{BazEqn1}
             q &= \sum_{\ell = 1}^5 |a_{\ell 5}|^2 q_\ell, \ \ {\rm and}\\
        \label{BazEqn2}
            0 &= \sum_{\ell = 1}^5 (|(Ak)_{\ell 2}|^2 + |(Ak)_{\ell 4} |^2)q_\ell
        \end{align}
respectively.
\end{lem}

\begin{proof}
We first recall that both $W_1$ and $W_2$ lie in $\k = \u(4)$.  Therefore $W_1$ and $W_2$ are horizontal with respect to $\met_1$ if and only if they are horizontal with respect to $\met_0$.  Moreover, $W_1$ and $W_2$ are both orthogonal to $\mf{sp}(2)$ with respect to the bi-invariant metric by our discussion above.  Hence we need only obtain expressions for $W_1$ and $W_2$ being orthogonal with respect to $\met_0$ to $v_A := \Ad_{A^*} Q - \diag(0, 0, 0, 0, iq)$, where $Q = \diag(i q_1, \dots, iq_5)$.

Recall that $\<X,Y\>_0 = -\Re \tr (XY)$.  Then $W_1$ is horizontal if and only if
        \begin{align*}
            -4 q &= \left\<\diag(0,0,0,0,iq), W_1 \right\>_0\\
                &= \<\Ad_{A^*} Q, W_1 \>_0\\
                &= \sum_{\ell = 1}^5 ( |a_{\ell 1}|^2 + |a_{\ell 2}|^2 + |a_{\ell 3}|^2 + |a_{\ell 4}|^2 - 4 |a_{\ell 5}|^2) q_\ell.
        \end{align*}
Now using the fact that $A$ is unitary together with $q = \sum_{\ell = 1}^5 q_\ell$ yields
        $$- 4 q = q - 5 \sum_{\ell = 1}^5 |a_{\ell 5}|^2 q_\ell$$
as desired.

Consider now $W_2 = \Ad_k \widehat W$, where $\widehat W = \diag(2i, -3i, 2i, -3i, 2i)$.  Then $W_2$ is horizontal if and only if
        \begin{align*}
            2 q &= \left\<\diag(0,0,0,0,iq), \widehat W \right\>_0\\
                &= \left\<\Ad_{k^*} \diag(0,0,0,0,iq), \widehat W \right\>_0 \ \ \ {\rm for} \ \ k \in \syp(2) \subset \SU(4)\\
                &= \left\<\diag(0,0,0,0,iq), W_2 \right\>_0\\
                &= \<\Ad_{A^*} Q, W_2 \>_0\\
                &= \left\<\Ad_{(Ak)^*} Q, \widehat W \right\>_0\\
                &= \sum_{\ell = 1}^5 \! \left(2 |(Ak)_{\ell 1}|^2 \! - 3 |(Ak)_{\ell 2}|^2
                          \! + 2 |(Ak)_{\ell 3}|^2 \! - 3 |(Ak)_{\ell 4}|^2 \! + 2 |(Ak)_{\ell 5}|^2\right) q_\ell\\
                &= \sum_{\ell = 1}^5 \left(2 - 5 \left(|(Ak)_{\ell 2}|^2 + |(Ak)_{\ell 4}|^2\right)\right) q_\ell, \ \ \ \text{since $A$ is unitary}.
        \end{align*}
Equation (\ref{BazEqn2}) now follows immediately from $q = \sum_{\ell = 1}^5 q_\ell$.\qed
\end{proof}


It is well-known (see \cite{Zi2}, \cite{DE}) that a general Bazaikin space $B^{13}_{q_1, \dots, q_5}$ admits positive curvature if and only if $q_i + q_j > 0$ for all $1 \leq i < j \leq  5$.  Since each $q_j$ is odd, it is clear that at least three of the $q_j$ must have the same sign.  Suppose that four of the $q_j$ share the same sign.  We may assume without loss of generality that $q_1, \dots, q_4$ are all positive.  We now prove Theorem \ref{thmA}\ref{QPBaz}.

\begin{thm}
\label{BazQP}
All $B^{13}_{q_1, \dots, q_5}$ with $q_1, \dots, q_4 > 0$ admit quasi-positive curvature.
\end{thm}

\begin{proof}
As we established in Lemmas \ref{Baz0curv} and \ref{Bazeqns}, there is a horizontal zero-curvature plane at $A \in \SU(5)$ if and only if we can solve either equation (\ref{BazEqn1}) or equation (\ref{BazEqn2}) at $A$.  If we allow $A$ to be diagonal then equations (\ref{BazEqn1}) and (\ref{BazEqn2}) become
        \begin{align}
        \label{BazQPeq1}
            q_5  &= \sum_{\ell = 1}^5 q_\ell, \ \ \ {\rm and}\\
        \label{BazQPeq2}
            0 &= \sum_{\ell = 1}^5 \left(|k_{\ell 2}|^2 + |k_{\ell 4}|^2\right) q_\ell
        \end{align}
respectively.  By hypothesis $q_1, \dots, q_4 > 0$ and therefore equality in (\ref{BazQPeq1}) is impossible.  On the other hand, because of how we have embedded $\syp(2)$ in $\SU(5)$, both $k_{52}$ and $k_{54}$ are zero.  Now since $k$ is unitary there are at least two non-zero coefficients $|k_{\ell 2}|^2 + |k_{\ell 4}|^2$, $\ell = 1, \dots 4$.  Therefore the right-hand side of equation (\ref{BazQPeq2}) is positive and thus no solutions exist.  We have shown there are no horizontal zero-curvature planes at diagonal $A \in \SU(5)$, which in turn implies the desired result.\qed
\end{proof}

It is natural to ask whether we can make a stronger curvature statement than quasi-positive curvature on the ``boundary'' of the positive curvature condition, namely when $q_i + q_j = 0$ for some $i,j$.  In fact, this is a rather restrictive condition.

\begin{lem}
\label{Bazmnfds}
Up to diffeomorphism, the spaces $B^{13}_{1,1,1,n,-n}$, $n \in \Z$ odd, describe all Bazaikin spaces satisfying $q_i + q_j = 0$ for some $1 \leq i < j \leq 5$.
\end{lem}

\begin{proof}
Recall that reordering the $q_i$ is a diffeomorphism, so we may assume without loss of generality that $q_4 + q_5 = 0$.  By the freeness condition (\ref{freeBaz}) we must have $q_i + q_j = \pm 2$ for $1 \leq i < j \leq 3$.  If we examine the eight possible combinations of these expressions we find that, up to sign and reordering, the only $5$-tuples which can arise are $(1,1,1,q_4,q_5)$ and $(1,1,-3,q_4,q_5)$, with $q_4 + q_5 = 0$.  However, following \cite{EKS}, Remark 4.2, and the discussion in Section 1 of \cite{FZ2}, we know that these $5$-tuples in fact describe the same manifolds.\qed
\end{proof}

With Theorem \ref{BazQP} and Lemma \ref{Bazmnfds} in hand, and recalling the situation for Eschenburg spaces (Theorem \ref{E0almostpos}), $B^{13}_{1,1,1,n,-n}$ provide a family of natural candidates to admit a metric with almost positive curvature.  In the case of $n=1$ we can indeed exhibit this property and in so doing we establish Theorem \ref{thmA}\ref{APBaz}.  For the cases $n>1$ the problem is open.

\begin{thm}
\label{BazCoh1AP}
The Bazaikin space $B^{13}_{1, 1, 1, 1, -1}$ admits almost positive curvature.
\end{thm}

\begin{proof}
Since $A$ is unitary, equation (\ref{BazEqn1}) becomes
        $$3 = \sum_{\ell=1}^4 |a_{\ell 5}|^2 - |a_{55}|^2 = 1 - 2|a_{55}|^2 < 1.$$
Therefore (\ref{BazEqn1}) has no solutions.  On the other hand, equation (\ref{BazEqn2}) becomes
        \begin{align*}
            0 &= \sum_{\ell=1}^4 (|(Ak)_{\ell 2}|^2 + |(Ak)_{\ell 4}|^2) - (|(Ak)_{5 2}|^2 + |(Ak)_{5 4}|^2) \\
            &= \sum_{\ell=1}^3 2(1 - (|(Ak)_{5 2}|^2 + |(Ak)_{5 4}|^2)).
        \end{align*}
Since $Ak$ is unitary, we know that $|(Ak)_{5 2}|^2 + |(Ak)_{5 4}|^2 \leq 1$.  Thus
        $$|(Ak)_{5 2}|^2 + |(Ak)_{5 4}|^2 = 1$$
and so $|(Ak)_{5 1}|^2 = |(Ak)_{5 3}|^2 = |(Ak)_{5 5}|^2 = 0$.  In particular $(Ak)_{55} = 0$.  But $(Ak)_{55} = a_{55} k_{55}$ because of our embedding of $\syp(2)$ in $\SU(5)$, and for the same reason $|k_{55}| = 1$.  Hence $a_{55} = 0$, and so $B^{13}_{1,1,1,1,-1}$ admits almost positive curvature since this is clearly invariant under the action of $\syp(2) \cdot S^1_{q_1, \dots, q_5}$.\qed
\end{proof}

\begin{rmk}
    As previously mentioned, one can find the argument for positive curvature on the Bazaikin spaces in \cite{Zi2} and \cite{DE}.  A proof following the modified argument used in this article may be found in \cite{Ke}.
\end{rmk}

Until recently only the integral cohomology ring of the Bazaikin spaces was known (\cite{Ba}).  Bazaikin spaces may be distinguished from one another via the order $s$ of the finite torsion groups $H^6 = H^8 = \Z_s$.  In \cite{FZ2} the authors give explicit expressions for $s$ and some other topological invariants.  In particular, the order $s$ and the first Pontrjagin class, $p_1$, are given by
        \begin{align}
        \label{cohomBaz}
            s &= \frac{1}{8} \ \Big|\sigma_3 \left(q_1, \dots, q_5, -\sum q_i \right)\Big|, \ \ {\rm and} \\
        \label{pontBaz}
            p_1 &= -\sigma_2 \left(q_1, \dots, q_5, -\sum q_i \right) \in H^4 = \Z,
        \end{align}
where $\sigma_i (a_1, \dots, a_r)$ denotes the elementary symmetric polynomial of degree $i$ in the variables $a_1, \dots, a_r$.  Note that $p_1$ is a homeomorphism invariant.  Theorem \ref{topBaz} now follows easily.

\begin{thm}
\label{Baztop}
The quasi-positively curved Bazaikin spaces $B^{13}_{1,1,1,n,-n}$, $n \geq 1$ odd, share the same cohomology ring but are pairwise homeomorphically distinct.
\end{thm}

\begin{proof}
Given (\ref{cohomBaz}) it is a simple exercise to compute that $s=1$ for each of the manifolds $B^{13}_{1,1,1,n,-n}$, whereas $p_1 = 6 + n^2$ by (\ref{pontBaz}).\qed
\end{proof}

\section{Torus quotients of $S^3 \x S^3$}
\label{4orbs}


Wilking, \cite{Wi}, has shown that a particular circle action on $S^3 \x S^3$ induces almost positive curvature on $S^3 \x S^2$.  This, together with the description in \cite{To} of $\C P^2 \# \C P^2$ as a biquotient $(S^3 \x S^ 3) \bq T^2$, suggests that it may be beneficial to study $T^2$ actions on $S^3 \x S^3$.  We are, of course, interested in finding new examples of biquotients with almost and quasi-positive curvature.  Recall that a bi-invariant metric on $S^3 \x S^3$ is simply a product of bi-invariant metrics on each factor.  Suppose we use a Cheeger deformation from the bi-invariant metric to equip $S^3 \x S^3$ with a left-invariant metric which is right-invariant under our $T^2$ action.  If we allow such isometric torus actions to be arbitrary on the right-hand side of $S^3 \x S^3$ then, since $\Im \H$ is $3$-dimensional, at every point of $S^3 \x S^ 3$ we will be able to obtain a horizontal zero-curvature plane of the form $\Span\{(v,0),(0,w) \ | \ v,w \in \Im \H \}$, which hence will project to a zero-curvature plane in $(S^3 \x S^ 3) \bq T^2$.  Therefore we shall restrict our attention to a special subfamily of torus actions which act arbitrarily on the left, but diagonally on the right of $S^3 \x S^3$.

Let $G=S^3 \times S^3$.  As we are interested in biquotient actions, we need to consider homomorphisms
        $$f :\  T^2 \!\!\! \lra T^2 \subset T^2 \x T^2 \subset G \x G$$
such that $f(T^2)$ is diagonal in the second factor, i.e. the projection onto the second factor is either trivial or one-dimensional.  Hence all tori $f(T^2)$ must have either one or two-dimensional projections onto the first factor.  If we perform the appropriate reparametrizations we may thus assume without loss of generality and up to a reordering of factors that the torus $f(T^2) \subset G \x G$ has one of the forms
        \begin{align}
            \label{zwTorus} U_L &:= \left\{ \left( \twoonem{z}{w}, \twoonem{1}{1} \right) \ \Big| \ z,w \in S^1 \right\}; \ \ \ {\rm or}\\
            \label{abTorus} U_{a,b} &:= \left\{ \left(\twoonem{z}{w}, \twoonem{z^a w^b}{z^a w^b} \right) \ \Big| \ z,w \in S^1 \right\}, \ \ \ a,b \in \Z; \ \ \ {\rm or}\\
            \label{cTorus} U_c &:= \left\{ \left(\twoonem{z}{z^c}, \twoonem{w}{w} \right) \ \Big| \ z,w \in S^1 \right\}, \ \ \ c \in \Z.
        \end{align}
It is clear that $U_L$ acts effectively and freely on $G$.  We are interested in determining when the other actions are free.

Consider $\H = \C + \C j$ and recall that $j z = \bar z j$ for all $z \in \C$.  Therefore, given some $q \in S^3 \subset \H$,
        \begin{align}
        \nonumber
             &{}  z^k w^\ell q \bar z^m \bar w^n = q   \\
        \nonumber
             &\iff  (z^k w^\ell) 1 (\bar z^m \bar w^n) = 1 \ \
               {\rm and }     \ \ (z^k w^\ell) j (\bar z^m \bar w^n) = j\\
        \label{eq1}
             &\iff  z^{k-m} w^{\ell-n} = 1 \ \  {\rm and }  \ \ z^{k+m} w^{\ell+n} = 1.
        \end{align}

It is a simple exercise using the equations in (\ref{eq1}) to show that $U_c$ and $U_{a,b}$ act effectively on $G$ when $c$ and $a + b$ respectively are even, while in the event that either $c$ or $a + b$ is odd there is an ineffective kernel $\Delta \Z_2 := \{\pm(1,1)\}$ for the respective action.

Moreover, one can easily check that the only points which can possibly be fixed by the actions of $U_c$ or $U_{a,b}$ (modulo any ineffective kernel) lie on the orbits of the points $(1,1)$, $(1,j)$, $(j,1)$ and $(j,j)$.  Therefore we need only examine these points in order to determine when the actions are free.

\begin{prop}
\label{freeT2}
Up to a change of coordinates or reordering of factors, the only free $T^2$ actions on $S^3 \times S^3$ which are diagonal on the right are given by $U_L = U_{a,b}$, $a=b=0$, and $U_c$, $c=0$.  The actions are, respectively,
        \begin{align*}
            (z,w) \star \twoonem{q_1}{q_2} &=  \twoonem{z q_1}{w q_2}, \ z,w \in S^1, q_1, q_2 \in S^3; \ \ {\rm and} \\
            (z,w) \star \twoonem{q_1}{q_2} &= \twoonem{z q_1 \bar w}{q_2 \bar w}, \ z,w \in S^1, q_1, q_2 \in S^3.
        \end{align*}
The resulting manifolds are both diffeomorphic to $S^2 \times S^2$.
\end{prop}

\begin{proof}
As the arguments are analogous, we consider only the action of $U_{a,b}$.  The $U_c$ case is left to the reader.  For $a+b$ even the equations in (\ref{eq1}) yield
        \begin{align*}
            (1,1) \ \ {\rm fixed} &\iff  z = w \ \ {\rm and} \ \ z^{1-a-b} = 1 ; \\
            (1,j) \ \ {\rm fixed} &\iff  z = \bar w \ \ {\rm and} \ \ z^{1-a+b} = 1; \\
            (j,1) \ \ {\rm fixed} &\iff  z = \bar w \ \ {\rm and} \ \ z^{1+a-b} = 1; \\
            (j,j) \ \ {\rm fixed} &\iff  z = w \ \ {\rm and} \ \ z^{1+a+b} = 1.
        \end{align*}
Thus we see that the action is free (namely $z=w=1$ in each case) if and only if $1 \pm a \pm b = \pm 1$.  But $a+b$ is even, hence $\pm a \pm b$ is even, and so $a=b=0$ is the only situation in which we can obtain a free action.

Suppose now that $a+b$ is odd.  The existence of a $\Delta \Z_2$ ineffective kernel implies that the action is free (namely $z = w = \pm 1$ in each case) if and only if $1 \pm a \pm b = \pm 2$.  It is a simple exercise to check that there are no values of $a$ and $b$ which satisfy all four equations simultaneously.  Hence we will always have a fixed point and so the action of $U_{a,b}$, $a+b$ odd, is never free.

The fact that the quotients under the free actions are diffeomorphic to $S^2 \x S^2$ follows from computing the cohomology ring.  Four dimensional manifolds with non-negative curvature that admit an effective, isometric circle action were classified in \cite{Kl} (see also \cite{SY}).  Only the manifolds $S^4$, $\C P^2$, $\C P^2 \# \pm \C P^2$ and $S^2 \x S^2$ can arise.  These manifolds are clearly distinguished by their cohomology rings.  To compute the cohomology of the biquotients under consideration one can follow the process described in \cite{Ke2} (see also \cite{dV}).\qed
\end{proof}

For those actions which are not free we may consider the equations obtained in the proof of Proposition \ref{freeT2} in order to write down explicitly the isotropy groups $\Gamma_{(q_1,q_2)}$ of singular points, which we recall can only be the $T^2$-orbits of the points $(q_1,q_2) = (1,1), (1,j), (j,1), (j,j) \in S^3 \x S^3$.  The isotropy groups for each action (modulo any ineffective kernel) are collected in Table \ref{isotropygroups}.  By considering the groups in this table we can easily find examples which have only one or two singular points and small isotropy groups at these points.  In the event that they arise, $\Z_0$ and $\Z_1$ denote $S^1$ and $\{1\}$ respectively.  We include some examples in Table \ref{specialisotropy}.

\begin{table}[!hbtp]
\begin{center}
\begin{tabular}{|l|c||c|c|c|c|}
    \cline{3-6} \multicolumn{2}{c||}{} & \multicolumn{4}{|c|}{$\Gamma_{(q_1,q_2)}$ at:}\\[.1cm] \cline{3-6}
    \multicolumn{2}{c||}{} & $(1,1)$ & $(1,j)$ & $(j,1)$ & $(j,j)$ \\[.1cm]
    \hline \hline
    \multirow{2}{*}{$U_{a,b}$} & $a+b$ even & $\Z_{|1-a-b|}$ & $\Z_{|1-a+b|}$ & $\Z_{|1+a-b|}$ & $\Z_{|1+a+b|}$\\[.1cm]
    \cline{2-6}
                           & $a+b$ odd & $\Z_{\frac{1}{2}|1-a-b|}$ & $\Z_{\frac{1}{2}|1-a+b|}$ & $\Z_{\frac{1}{2}|1+a-b|}$ & $\Z_{\frac{1}{2}|1+a+b|}$\\[.1cm]
    \hline
    \multirow{2}{*}{$U_c$} & $c$ even & $\Z_{|c-1|}$ & $\Z_{|c+1|}$ & $\Z_{|c+1|}$ & $\Z_{|c-1|}$\\[.1cm]
    \cline{2-6}
                           & $c$ odd & $\Z_{\frac{1}{2}|c-1|}$ & $\Z_{\frac{1}{2}|c+1|}$ & $\Z_{\frac{1}{2}|c+1|}$ & $\Z_{\frac{1}{2}|c-1|}$\\[.1cm]
    \hline
\end{tabular}
\end{center}
\caption{Isotropy groups of the $T^2$ actions $U_{a,b}$ and $U_c$} \label{isotropygroups}
\end{table}

\begin{table}[!hbtp]
\begin{center}
\begin{tabular}{|l||c|c|c|c|} \cline{2-5}
    \multicolumn{1}{c||}{} & \multicolumn{4}{|c|}{$\Gamma_{(q_1,q_2)}$ at:}\\ \cline{2-5}
    \multicolumn{1}{c||}{} & $(1,1)$ & $(1,j)$ & $(j,1)$ & $(j,j)$\\
    \hline \hline
                     $U_{1,1}$ & $\{1\}$ & $\{1\}$ & $\{1\}$ & $\Z_3$ \\ \hline
                     $U_{3,0}$ & $\{1\}$ & $\{1\}$ & $\Z_2$ & $\Z_2$ \\ \hline
                     $U_2$ & $\{1\}$ & $\Z_3$ & $\Z_3$ & $\{1\}$\\ \hline
                     $U_3$ & $\{1\}$ & $\Z_2$ & $\Z_2$ & $\{1\}$\\ \hline
\end{tabular}
\end{center}
\caption{Some special cases of the actions $U_{a,b}$ and $U_c$} \label{specialisotropy}
\end{table}


We turn now to the curvature computations.  Consider the subgroup $K=\Delta S^3 \subset G=S^3 \times S^3$, and let $\<\ ,\>_0$ be the bi-invariant product metric on $G$.  Then $\g = \k \oplus \p$, where $\p$ is the orthogonal complement to $\k$ with respect to $\<\ ,\>_0$.  Notice that $(G,K)$ is a rank one symmetric pair.  We define a new left-invariant, right $K$-invariant metric $\met_1$ on $G$ as in (\ref{met1}), namely
        $$\<X,Y\>_1 = \<X,\Phi(Y)\>_0,$$
where $\Phi(Y) = Y_\p + \lambda Y_\k$, $\lambda \in (0,1)$.  By Lemma \ref{Eschlem} we know that a plane
$\sigma = \Span\{\Phi^{-1}(X),\Phi^{-1}(Y)\} \subset \g$ has zero-curvature with respect to $\<\ ,\>_1$ if and only if
        $$0=[X,Y]=[X_\p,Y_\p]=[X_\k,Y_\k].$$
Hence, for $G$ and $K$ as above, a zero-curvature plane must be of the form
    \beq
    \label{0curv}
        \sigma = \Span\{\Phi^{-1}(v,0), \Phi^{-1}(0,v) \ | \ v \in \Im \H\}.
    \eeq
Since we are considering $T^2$ actions which are diagonal on the right of $G$, it is clear that the actions are by isometries and hence induce a metric on $G \bq T^2$.
\begin{thm}
\label{almostposcurv}
    $(G, \<\ ,\>_1) \bq T^2$ has almost positive curvature if and only if the action is not free.
\end{thm}

\begin{proof}
By O'Neill's formula it is sufficient to show that points in $G$ with horizontal zero-curvature planes lie on a hypersurface.  Recall that the existence of an ineffective kernel will have no impact on our curvature computations.  We therefore need only consider torus actions of the form
        \begin{align*}
            U_{a,b} &= \left\{ \left(\twoonem{z}{w}, \twoonem{z^a w^b}{z^a w^b} \right) \ \Big| \ z,w \in S^1 \right\}, \ \ \ a,b \in \Z;\\
            U_c &= \left\{ \left(\twoonem{z}{z^c}, \twoonem{w}{w} \right) \ \Big| \ z,w \in S^1 \right\}, \ \ \ c \in \Z,
        \end{align*}
and notice that the $U_L$ action of (\ref{zwTorus}) is the special case $(a,b)=(0,0)$ of $U_{a,b}$.

Consider first the action by $U_{a,b}$.  The vertical subspace at $(q_1,q_2)$, left translated to $(1,1)$, is given by
        $$\mathcal{V}_{(q_1,q_2)} = \left\{\frac{1}{2} \twoonem{s \Ad_{\bar q_1} i -(a \: s + b \: t) i}{t \Ad_{\bar q_2} i -(a \: s + b \: t)  i}
            \ \Bigg| \ s, t \in \R\right\}.$$
Thus the horizontal subspace with respect to $\<\ , \>_1$ is
        $$\mathcal{H}_{(q_1,q_2)} = \left\{\Phi^{-1}(v,w) \ \Bigg| \ \begin{matrix} \Ad_{q_1} v - a(v+w) \perp i\\
                                                                     \Ad_{q_2} w - b(v+w)\perp i \end{matrix} \right\}.$$
Hence, by equation (\ref{0curv}), a zero-curvature plane
        $$\sigma = \Span\{\Phi^{-1}(v,0), \Phi^{-1}(0,v)\}$$
is horizontal if and only if
            $$\Ad_{q_1} v - a v \perp i,\ \
            a v \perp i,\ \
            \Ad_{q_2} v - b v \perp i \ \ {\rm and } \ \
            b v \perp i.$$
We want to show that $v, \Ad_{q_1} v, \Ad_{q_2} v \perp i$ since this is equivalent to $v \perp i, \Ad_{\bar q_1} i$, $\Ad_{\bar q_2} i$.  This will imply that $v=0$ unless $i, \Ad_{\bar q_1} i$, and $\Ad_{\bar q_2} i$ are linearly dependent, which in turn would imply positive curvature at the point $[(q_1,q_2)] \in G \bq U_{a,b}$.  It is clear that this situation arises if and only if $(a,b) \neq (0,0)$, i.e. if and only if the action of $U_{a,b}$ is not free.  Suppose $(a,b) \neq 0$.  Then $i, \Ad_{\bar q_1} i$, and $\Ad_{\bar q_2} i$ are linearly dependent if and only if
        \beq
        \label{detcond}
            \det \twom{\<\Ad_{\bar q_1} i, j\>}{\<\Ad_{\bar q_1} i, k\>}{\<\Ad_{\bar q_2} i, j\>}{\<\Ad_{\bar q_2} i, k\>} = 0,
        \eeq
which defines a hypersurface in $G$.  Note that equation (\ref{detcond}) is invariant under the action of $U_{a,b}$ since
$\Ad_{z^k w^\ell q \bar z^m \bar w^n} i = \Ad_{z^k w^\ell q} i$ and $\<\Ad_q i, j\> = 2\Re(\bar u v i)$, $\<\Ad_q i, k\> = 2\Re(\bar u v)$,
for $z,w \in S^1$, $q = u + vj \in S^3, u,v \in \C$.  Thus we have a hypersurface in $G \bq U_{a,b}$ defined by (\ref{detcond}) on which points with horizontal zero-curvature planes must lie.

We now turn our attention to the action by $U_c$.  The vertical subspace at $(q_1,q_2)$, left translated to $(1,1)$, is given by
        $$\mathcal{V}_{(q_1,q_2)} = \left\{\frac{1}{2} \twoonem{s \Ad_{\bar q_1} i - t \: i}{c \: s \Ad_{\bar q_2} i - t \:  i}
            \ \Bigg| \ s, t \in \R\right\}.$$
Thus the horizontal subspace with respect to $\<\ , \>_1$ is
        $$\mathcal{H}_{(q_1,q_2)} = \left\{\Phi^{-1}(v,w) \ \Bigg| \ \begin{matrix} \Ad_{q_1} v + c \Ad_{q_2} w \perp i\\
                                                                    v+w \perp i \end{matrix} \right\}.$$
Hence, by (\ref{0curv}), a zero-curvature plane $\sigma = \Span\{\Phi^{-1}(v,0), \Phi^{-1}(0,v)\}$ is horizontal if and only if
        $$\Ad_{q_1} v \perp i,\quad
            c \Ad_{q_2} v \perp i \ \ {\rm and} \ \
            v \perp i.$$
It is clear that the only situation in which we do not get $v, \Ad_{q_1} v, \Ad_{q_2} v \perp i$ is when $c = 0$, i.e. when the action is free.  In all other situations we have almost positive curvature by the same argument as for $U_{a,b}$.\qed
\end{proof}

\begin{rmk}
In each case there is a horizontal zero-curvature plane at $(q_1,q_2)$ if one of the following holds:
\begin{enumerate}
    \item $\Ad_{\bar q_1} i = \pm i, \textrm{ i.e.} \ \ q_1 \in \C \ {\rm or} \ \C j$
    \item $\Ad_{\bar q_2} i = \pm i, \textrm{ i.e.} \ \ q_2 \in \C \ {\rm or} \ \C j$
    \item $\Ad_{\bar q_1} i = \pm \Ad_{\bar q_2} i, \textrm{ i.e.} \ \ q_1 \perp q_2, i q_2 \ {\rm or} \ q_1 \perp j q_2, k q_2$.
\end{enumerate}
Thus we will always have a zero-curvature plane at the singular points when the action is not free.  Moreover, in the free cases we have a zero-curvature plane at every point.  More precisely:
\begin{itemize}
    \item The action $U_{a,b}$ with $a=b=0$ yields $\Ad_{\bar q_1} i, \Ad_{\bar q_2} i \perp v$, which implies that there is a unique horizontal zero-curvature plane when $\Ad_{\bar q_1} i$ and $ \Ad_{\bar q_2} i$ are linearly independent, and there is an $S^1$ worth of zero-curvature planes when $\Ad_{\bar q_1} i = \pm \Ad_{\bar q_2} i$, i.e. when $q_1 \perp q_2, i q_2 \ {\rm or} \ q_1 \perp j q_2, k q_2$.
    \item The action $U_c$ with $c=0$ yields $i, \Ad_{\bar q_1} i \perp v$, which implies that there is a unique horizontal zero-curvature plane when $q_1 \not\in \C \ {\rm or} \ \C j$, and there is an $S^1$ worth of zero-curvature planes when $q_1 \in \C \ {\rm or} \ \C j$.
\end{itemize}
\end{rmk}

\begin{rmk} It is not difficult to show that on each of the orbifolds above there is, up to reparametrization, a unique non-trivial isometric circle action.  The image of the loci of points admitting a zero-curvature plane in the corresponding $S^1$-orbit space is topologically a two-dimensional sphere.
\end{rmk}

\begin{ack}
The majority of this work was completed as part of the author's Ph.D. thesis at the University of Pennsylvania.  It is with great pleasure that thanks are extended to Wolfgang Ziller for his patience, support, friendship and guidance, and to Dan Jane for reading a preliminary version of this article and providing numerous comments.
\end{ack}

\end{document}